\documentclass[10pt,letterpaper]{article}
\usepackage[scale=0.7]{geometry}
\usepackage{amsmath,amssymb,amsthm}
\usepackage{enumitem}
\usepackage{microtype}
\usepackage[colorlinks=true,linkcolor=blue,citecolor=blue,urlcolor=blue]{hyperref}

\newtheorem{theorem}{Theorem}[section]
\newtheorem{lemma}[theorem]{Lemma}
\newtheorem{corollary}[theorem]{Corollary}
\newtheorem{claim}{Claim}
\theoremstyle{remark}

\newcommand{\aTwo}{\alpha_2}
\newcommand{\aZero}{\alpha_0}
\newcommand{\E}{\mathbb E}
\newcommand{\Prob}{\mathbb P}
\newcommand{\CE}{Chv\'atal--Erd\H{o}s}
\newcommand{\abs}[1]{\lvert#1\rvert}

\newcommand{\floor}[1]{\left\lfloor#1\right\rfloor}
\setlist[enumerate]{label=(\roman*),leftmargin=*,topsep=4pt,itemsep=2pt}
\hypersetup{pdftitle={A quadratic refinement of Jackson's Chvatal--Erdos condition for Hamilton cycles in digraphs},pdfauthor={Jiangdong Ai and Yongtang Shi}}

\title{A quadratic refinement of Jackson's \CE\ condition\\
for Hamilton cycles in digraphs}
\author{Jiangdong Ai\thanks{School of Mathematical Sciences and LPMC, Nankai University.
\texttt{jd@nankai.edu.cn}. Partially supported by the National Natural Science Foundation of China
(No. 12522117) and Fundamental and Interdisciplinary Disciplines Breakthrough Plan of the Ministry
of Education of China (JYB2025XDXM207).}\quad
Yongtang Shi\thanks{Center for Combinatorics and LPMC, Nankai University, Tianjin 300071, China.
\texttt{shi@nankai.edu.cn}. Partially supported by the National Natural Science Foundation of China
(No. 12431013).}}
\date{}

\begin{document}
\maketitle

\begin{abstract}
For a digraph $D$, let $\aTwo(D)$ be the largest size of a vertex set no two of whose vertices
lie in a common directed $2$-cycle. Let $f_2(a)$ be the least integer $K$ such that every
$K$-connected digraph $D$ with $\aTwo(D)\le a$ has a Hamilton cycle. Jackson proved in 1987
that $f_2(a)\le2^a(a+2)!$, whereas a conjecture of Jackson and Ordaz predicts
$f_2(a)\le a+1$. We prove the quadratic bound $f_2(a)\le12000a^2$. We also prove that
$\kappa(H)\ge6000(\alpha(H)+r)$ guarantees vertex-disjoint paths joining any prescribed
$r$ pairs of distinct vertices and covering $V(H)$.
\end{abstract}



\section{Introduction}

One of the best known sufficient conditions for Hamiltonicity is the theorem of Chv\'atal
and Erd\H{o}s~\cite{CE}: a $2$-connected graph $G$ with $\kappa(G)\ge\alpha(G)$ has a
Hamilton cycle, that is, a cycle through all of its vertices. If moreover
$\kappa(G)\ge\alpha(G)+1$, then $G$ is Hamilton-connected, meaning that every two distinct
vertices are the endpoints of a Hamilton path. We refer to~\cite{BG} for terminology and
notation not introduced here. All graphs and digraphs are finite, without loops or multiple
arcs; in a digraph, Hamilton cycles and paths are required to be directed.

A directed analogue requires a directed substitute for $\alpha$. For a digraph $D$, let
$G(D)$ be the graph on $V(D)$ in which $uv$ is an edge precisely when $D$ contains both
$u\to v$ and $v\to u$, and let $\aTwo(D)=\alpha(G(D))$. Thus $\aTwo(D)$ is the largest
number of vertices no two of which lie in a common directed $2$-cycle. Replacing each edge
of a graph $G$ by a $2$-cycle produces a digraph $D$ with $\aTwo(D)=\alpha(G)$, so $\aTwo$
extends $\alpha$. We also use the smaller parameter $\aZero(D)$, the largest size of a
vertex set spanning no arc. Clearly $\aZero(D)\le\aTwo(D)$.

We say that a digraph is strongly connected if it contains a directed path from each vertex
to every other one. Throughout, a digraph $D$ is $K$-connected if $\abs{V(D)}\ge K+1$ and
$D-X$ is strongly connected for every $X\subseteq V(D)$ with $\abs X<K$. We write $k(D)$
for the largest such $K$, and use $\kappa(G)$ for the vertex-connectivity of an undirected
graph. In particular, a complete symmetric digraph of order $n$ has connectivity $n-1$.

Equality of connectivity and $\aTwo$ does not suffice: Thomassen~\cite{Thomassen} and
Chakroun~\cite{Chakroun} constructed non-Hamiltonian $j$-connected digraphs with
$\aTwo=j$ for $j=2$ and $j=3$, respectively; see~\cite[3.2.1--3.2.2]{JO}. Let
$f_2(a)$ be the least integer $K$ such that every $K$-connected digraph $D$ with
$\aTwo(D)\le a$ is Hamiltonian. Jackson and Ordaz~\cite{JO2} conjectured that this
function exists (see also~\cite[Problem~3.2.4]{JO}), and Jackson~\cite[Theorem~1]{Jackson}
confirmed this by proving $f_2(a)\le2^a(a+2)!$. The exact values $f_2(1)=1$ and
$f_2(2)=3$ were recorded by Jackson and Ordaz~\cite{JO2}, and $f_2(3)=4$ was
determined by Chakroun~\cite{Chakroun} and Chakroun and Sotteau~\cite{CS};
see~\cite[3.2.7--3.2.9]{JO}. In~\cite[Problem~4.1]{Jackson}, Jackson asked for
better bounds on $f_2(a)$ and observed that a linear function was conceivable.

Jackson and Ordaz~\cite[Conjecture~3.4.4]{JO} conjectured that every digraph $D$
with $k(D)\ge\aTwo(D)+1$ is pancyclic, that is, has cycles of every length from $2$
to $\abs{V(D)}$; this holds for $\aTwo(D)\le3$ by results of Thomassen and of
Chakroun and Sotteau, see~\cite[Section~3.4]{JO}. Its Hamiltonian consequence,
restated by K\"uhn and Osthus~\cite[Conjecture~11]{KO}, reads
$$
k(D)\ge\aTwo(D)+1\quad\Longrightarrow\quad D\text{ has a Hamilton cycle}.
$$
In terms of $f_2$, this predicts $f_2(a)\le a+1$, and the values above show that
this would be tight for $a=2,3$. Jackson and Ordaz~\cite[Problem~3.2.3]{JO} also
asked whether there is a non-Hamiltonian digraph $D$ with $k(D)\ge\aTwo(D)\ge4$;
see also~\cite{KO}. K\"uhn and Osthus~\cite{KO} remarked that even a polynomial
bound for $f_2$ would be interesting. The cubic bound $f_2(a)\le2a^3+2$ was obtained
in~\cite{AS}. In this paper we improve the order of that upper bound from cubic to
quadratic.

\begin{theorem}\label{thm:main}
For every integer $a\ge1$, $f_2(a)\le12000a^2$; that is, every $12000a^2$-connected
digraph $D$ with $\aTwo(D)\le a$ has a Hamilton cycle.
\end{theorem}

For $a\ge2$, our proof gives the slightly stronger sufficient connectivity
\begin{equation}\label{eq:precise-main}
14a+(a-1)(12000a-1)=12000a^2-11987a+1.
\end{equation}
The constants are not optimized. The lower bound $f_2(a)\ge a$ is noted
in~\cite[Problem~4.1]{Jackson} and~\cite[3.2.6]{JO}: for $a\ge2$, replacing every
edge of $K_{a-1,a}$ by a $2$-cycle gives a digraph with connectivity $a-1$ and
$\aTwo=a$ that has no Hamilton cycle, because its two parts have different orders.
Thus Theorem~\ref{thm:main} does not determine the correct order of $f_2$.

We need the following undirected result, which may be of independent interest.

\begin{theorem}\label{thm:spanning}
Let $a,r\ge1$ be integers, and let $H$ be an undirected graph with $\alpha(H)\le a$ and
$\kappa(H)\ge6000(a+r)$. For any $2r$ distinct vertices $x_1,y_1,\ldots,x_r,y_r$ of $H$,
there are pairwise vertex-disjoint paths $P_1,\ldots,P_r$ such that $P_i$ has endpoints
$x_i,y_i$ for every $i$, and $V(P_1)\cup\cdots\cup V(P_r)=V(H)$.
\end{theorem}

Aigner-Horev, Hefetz and Krivelevich~\cite[Theorem~2.15]{AHK} proved the same conclusion
under the condition $\kappa(H)\ge c\max\{\alpha(H),\log\abs{V(H)},r\}$ for an absolute
constant $c$. An order-independent sufficient condition was obtained by Jedli\v{c}kov\'a
and Kratochv\'il~\cite[Theorem~7]{JK}: for $\alpha(H)\le a$, connectivity
$r\max\{a+1,10\}$ suffices. Their work was incorporated into the joint paper of Fomin,
Golovach, Jedli\v{c}kov\'a, Kratochv\'il, Sagunov and Simonov~\cite{FGJKSS}.
The prescribed spanning-subdivision result in~\cite[Proposition~3]{FGJKSS}, applied
to a matching of $r$ edges, gives the sufficient condition
$\kappa(H)\ge3r\max\{a+3,10\}$.

Theorem~\ref{thm:spanning} gives an additive dependence on $a$ and $r$, with no term
involving the order of $H$. It is not a uniform numerical improvement over these
bounds. More precisely, $6000(a+r)$ is strictly smaller than
$r\max\{a+1,10\}$ exactly when
$a\ge6000$ and $r>\frac{6000a}{a-5999}$, and strictly smaller than $3r\max\{a+3,10\}$ exactly when
$a\ge1998$ and $r>\frac{2000a}{a-1997}.$
On the diagonal $a=r$, these crossovers occur at $a=12000$ and $a=3998$,
respectively. The constants are not optimized. These comparisons concern sufficient
connectivity for spanning linkages, not the algorithmic results of~\cite{JK,FGJKSS}.

The additive dependence is the point needed for Theorem~\ref{thm:main}. The final
reduction has a partition of total independence weight $\beta\le2a$. If there are
$m$ parts and one has independence number $\nu_i$, then
$\nu_i+m-1\le\beta$. Our linking lemma consequently needs only $O(a)$ connectivity
inside each noncomplete part, rather than a product of its independence number
and the number of parts. At most $a-1$ cuts of size $O(a)$ are then removed, giving
a total cost of $O(a^2)$. This explains the change from the cubic bound in~\cite{AS}
to the quadratic bound here.

The sampling estimate used to prove Theorem~\ref{thm:spanning} is stated separately.
Here and below, $K$ denotes an integer connectivity lower bound, while $k(D)$ is
the strong connectivity of a digraph.

\begin{theorem}\label{thm:sampling}
Let $a,K$ be positive integers, and let $G$ be an undirected graph with
$\alpha(G)\le a$ and $\kappa(G)\ge K\ge512a$. Include each vertex in $S$
independently with probability $1/2$, and let $h=\floor{K/512}$. Then
$$
\Prob\bigl(\kappa(G[S])<h\bigr)\le2e^{-K/64}<1/16.
$$
\end{theorem}

The general sampling theorem of Censor-Hillel, Ghaffari, Giakkoupis, Haeupler and
Kuhn~\cite[Theorem~1]{Sampling} requires sampling probability
$p\ge c_0\sqrt{\log n/K}$ for an absolute constant $c_0$, and gives remaining
connectivity $\Omega(Kp^2)$. Its half-sampling case is used in the proof of
\cite[Theorem~2.15]{AHK}, via their Theorem~2.13; it introduces a requirement
$K\ge c_1\log n$. That proof also uses a union bound over all $n$ vertices to
control their neighbourhoods in the sampled parts. Our use of the Caro--Wei
bound~\cite{Caro,Wei} removes the order dependence from this neighbourhood estimate,
and the fact that an induced subgraph has at most $\alpha(G)$ components gives
an order-independent connected-domination estimate. We retain the random-colour
deletion argument of~\cite[Section~2.1]{Sampling}. The final two-part completion
follows~\cite{AHK}.

In Section~3 we prove Theorems~\ref{thm:sampling} and~\ref{thm:spanning}. In Section~4 we give
the linking lemma; its complete-part case is recorded as
Corollary~\ref{cor:cliques}, with the linear condition $k(D)\ge7m$.
In Section~5 we combine the linking lemma with the fixed-threshold decomposition
and path contraction from~\cite{AS}. All parts of that reduction are included.

\section{Preliminaries}

A digraph is symmetric if every arc lies in a directed $2$-cycle. Equivalently, it is
obtained from an undirected graph by replacing every edge by two opposite arcs. For a
symmetric digraph $J$, we have $k(J)=\kappa(G(J))$ and $\aTwo(J)=\alpha(G(J))$. When a vertex set $V_i\subseteq V(D)$ is assigned a symmetric subdigraph $H_i$,
we mean that $V(H_i)=V_i$ and $A(H_i)\subseteq A(D[V_i])$, with each arc of $H_i$
accompanied by its reverse. The subdigraph $H_i$ need not be induced: further
arcs of either direction in $D[V_i]$ are allowed.

A path cover consists of vertex-disjoint directed paths covering all vertices; single
vertices count as paths. The Gallai--Milgram theorem~\cite{GM} states that every digraph
$D$ has a path cover of size at most $\aZero(D)$. A cycle factor is a spanning collection
of vertex-disjoint directed cycles. Directed $2$-cycles are allowed, but single vertices
are not cycles.

An undirected graph is $r$-linked if it has at least $2r$ vertices and any $r$
prescribed pairs of $2r$ distinct vertices can be joined by pairwise vertex-disjoint
paths. If the paths can always be chosen to cover all vertices, the graph is called
Hamiltonian-$r$-linked~\cite{JK,FGJKSS}. We use the theorem of Thomas and
Wollan~\cite{TW} that every $10r$-connected graph is $r$-linked. These paths need not
cover the graph. We also use the Hamilton-connected form of the \CE\ theorem stated
in the introduction. Both results will be applied only to undirected graphs.

For an undirected graph $G$, write $N_G[v]=N_G(v)\cup\{v\}$ and $d_G(v)=\abs{N_G(v)}$.
A connected dominating set is a nonempty vertex set $A$ such that $G[A]$ is connected
and every vertex outside $A$ has a neighbour in $A$. For $X\subseteq V(D)$, let
$N_D^+(X)$ be the set of all heads of arcs with tails in $X$; this set may meet $X$.
For a graph or digraph $J$, we also write $\abs J=\abs{V(J)}$. The empty graph and a single
vertex have connectivity zero. All logarithms are natural. The elementary deletion inequalities
$\kappa(G-W)\ge\kappa(G)-\abs W$ and $k(D-W)\ge k(D)-\abs W$ will be used whenever
the right-hand side is positive; their order requirements follow from those of $G,D$.
Our connectivity convention also implies $d_D^+(v),d_D^-(v)\ge k(D)$ for every
vertex $v$. For example, if $d_D^-(v)<k(D)$, deleting $N_D^-(v)$ leaves both $v$
and another vertex, by $\abs{V(D)}\ge k(D)+1$, and prevents any path from the
latter to $v$. The outdegree argument is the same.

We first recall the Caro--Wei bound~\cite{Caro,Wei} and record the elementary
estimates needed below.

\begin{lemma}\label{lem:degree}
Let $G$ be a graph with $\alpha(G)\le a$. Then
\begin{equation}\label{eq:reciprocal}
\sum_{v\in V(G)}\frac1{d_G(v)+1}\le\alpha(G)\le a.
\end{equation}
If $K\ge0$, $\lambda>0$, $\delta(G)\ge K$ and $\lambda(K+1)\ge1$, then
\begin{equation}\label{eq:exponential}
\sum_{v\in V(G)}e^{-\lambda(d_G(v)+1)}
\le a(K+1)e^{-\lambda(K+1)}.
\end{equation}
Every induced subgraph of $G$ has at most $a$ connected components.
\end{lemma}

\begin{proof}
The Caro--Wei bound gives~\eqref{eq:reciprocal}. Since $\lambda(K+1)\ge1$, the function
$xe^{-\lambda x}$ is nonincreasing for $x\ge K+1$, so
$$
\sum_v e^{-\lambda(d_G(v)+1)}
=\sum_v\frac{(d_G(v)+1)e^{-\lambda(d_G(v)+1)}}{d_G(v)+1}
\le(K+1)e^{-\lambda(K+1)}\sum_v\frac1{d_G(v)+1}.
$$
Combining this with~\eqref{eq:reciprocal} gives~\eqref{eq:exponential}. Finally,
representatives from different connected components of an induced subgraph form an
independent set in $G$.
\end{proof}

We next show a directed matching observation.

\begin{lemma}\label{lem:factor}
Let $D$ be a digraph, and let $F\subseteq A(D)$ have $f$ arcs, with at most one arc
entering and at most one arc leaving each vertex. If $k(D)\ge\aZero(D)+f$,
then $D$ has a cycle factor containing $F$.
\end{lemma}

\begin{proof}
Write $\lambda=k(D)$, $b=\aZero(D)$ and $n=\abs{V(D)}$. We first prove that every
nonempty set $X\subseteq V(D)$ satisfies
\begin{equation}\label{eq:expansion}
\abs{N_D^+(X)}\ge\min\{n,\abs X+\lambda-b\}.
\end{equation}
If $N_D^+(X)=V(D)$, this is immediate. Otherwise let
$I=X\setminus N_D^+(X)$ and $U=N_D^+(X)\setminus X.$
There is no arc from $X$ to $I$, so $I$ spans no arc and $\abs I\le b$.
If a vertex lies outside $X\cup N_D^+(X)$, deleting $U$ prevents a path from
$X$ to that vertex; both $X$ and that vertex survive, so $\abs U\ge\lambda$.
Otherwise $I\ne\varnothing$, and every in-neighbour of a vertex $v\in I$
lies in $U$. By the minimum indegree observation above,
$\abs U\ge d_D^-(v)\ge\lambda$. Consequently
$$
\abs{N_D^+(X)}-\abs X=\abs U-\abs I\ge\lambda-b,
$$
which proves~\eqref{eq:expansion} in the remaining case.

Let $T$ and $Z$ be the sets of tails and heads of the arcs of $F$, respectively.
Both have size $f$. Form a bipartite graph $\mathcal B$ with vertex classes
$$
L=\{v_L:v\in V(D)\setminus T\},\qquad
R=\{v_R:v\in V(D)\setminus Z\},
$$
where $L$ and $R$ are disjoint copies of the indicated vertex sets. Join $u_L$
to $v_R$ if and only if $u\to v\in A(D)$. For any nonempty
$X\subseteq V(D)\setminus T$, let $X_L=\{x_L:x\in X\}$. Then
\begin{align*}
\abs{N_{\mathcal B}(X_L)}
&=\abs{N_D^+(X)\setminus Z}\\
&\ge\min\{n,\abs X+\lambda-b\}-f\\
&=\min\{n-f,\abs X+\lambda-b-f\}\\
&\ge\abs X=\abs{X_L}.
\end{align*}
The last inequality uses $\abs X\le n-f$ and $\lambda\ge b+f$.
The empty set also satisfies Hall's condition. Since $\abs L=\abs R=n-f$,
Hall's theorem gives a perfect matching $\mathcal M$ of $\mathcal B$.

Let $F'=\{u\to v:u_Lv_R\in\mathcal M\}$. A vertex in $T$ has exactly one
outgoing arc in $F$ and none in $F'$, while a vertex outside $T$ has exactly
one outgoing arc in $F'$ and none in $F$. The analogous statement holds for
incoming arcs, using $Z$. Thus every vertex has indegree and outdegree one
in the spanning subdigraph with arc set $F\cup F'$. Its successor map is a
permutation, whose cycles form a cycle factor containing $F$.
No cycle has length one, since $D$ has no loops.
\end{proof}

\section{Spanning paths with prescribed ends}

We first bound the probability that a random vertex set fails to be connected and
dominating. We use the exponential consequence of the Caro--Wei bound in
Lemma~\ref{lem:degree}, together with the bound on the number of components of an
induced subgraph, to avoid a factor depending on the order of the graph.

\begin{lemma}\label{lem:quarter}
Let $a,K$ be positive integers, and let $G$ be a graph with $\alpha(G)\le a$ and
$\kappa(G)\ge K\ge8$. Include every vertex
in $T$ independently with probability $1/4$. Then
\begin{equation}\label{eq:quarter}
\Prob\bigl(T\text{ is not a connected dominating set of }G\bigr)
\le a(K+1)e^{-(K+1)/8}+2^a e^{-K/49}.
\end{equation}
\end{lemma}

\begin{proof}
We first observe that, if $A$ dominates $G$ and $A=C\dot\cup B$ is a partition into
nonempty unions of components of $G[A]$, then there are at least $K$ $C$--$B$ paths
whose internal vertices lie outside $A$, with pairwise disjoint interiors and at most
two internal vertices on each path. Endpoints of different paths may coincide.

Indeed, partition $V(G)\setminus A$ into $Z,U,W$, where vertices of $Z$ have neighbours
in both $C$ and $B$, vertices of $U$ have neighbours in $C$ but none in $B$, and
vertices of $W$ have neighbours in $B$ but none in $C$. These sets cover the complement
because $A$ is dominating. Every vertex of $Z$ supplies a one-internal-vertex path.
A maximum matching in the bipartite graph of edges between $U$ and $W$ supplies
paths with two internal vertices. All these interiors are disjoint. If their number
were less than $K$, K\H{o}nig's theorem would give a vertex cover $R$ of the $U$--$W$
bipartite graph with $\abs Z+\abs R<K$. Deleting $Z\cup R$ leaves no edge between the
nonempty sets $C\cup(U\setminus R)$ and $B\cup(W\setminus R)$, a contradiction.

Generate $T$ in two steps. First include each vertex in $A$ independently with
probability $1/8$, and then include each vertex outside $A$ independently with
probability $1/7$. The final inclusion probability is $1/8+(7/8)(1/7)=1/4$, independently
for different vertices. Since $\delta(G)\ge K$, Lemma~\ref{lem:degree} gives
\begin{equation}\label{eq:domination}
\Prob(A\text{ does not dominate }G)
\le\sum_v(7/8)^{d_G(v)+1}
\le a(K+1)e^{-(K+1)/8}.
\end{equation}

Condition on any dominating first set $A$. Its induced graph has at most $a$ components.
For every nontrivial bipartition of these components, choose $K$ paths as above before
exposing the second set. A path is retained when all its internal vertices are included
in the second step, an event of probability at least $1/49$. For the $K$ paths belonging
to one fixed bipartition, these events are independent. Thus the probability that none
is retained is at most $(1-1/49)^K\le e^{-K/49}$.

There are at most $2^a$ bipartitions to consider. With conditional probability at least
$1-2^ae^{-K/49}$, all of them receive a retained path. Every newly included vertex has
a neighbour in $A$, so every component of $G[T]$ contains a component of $G[A]$. If
$G[T]$ were disconnected, its components would determine a nontrivial bipartition
of the old components with no retained connecting path. Hence $G[T]$ is connected.
It is also dominating, since it contains $A$. Combining this with~\eqref{eq:domination}
proves~\eqref{eq:quarter}. Independence between different bipartitions is not required.
\end{proof}

To pass from connectedness to vertex-connectivity, we use the colouring argument of
Censor-Hillel, Ghaffari, Giakkoupis, Haeupler and Kuhn~\cite[Section~2.1]{Sampling}.
The estimate needed for each choice of colours is supplied by Lemma~\ref{lem:quarter}.

\begin{proof}[Proof of Theorem~\ref{thm:sampling}]
Independently give each vertex of $S$ one of $2h$ colours, uniformly at random. For each
$h$-element subset $Q$ of the colour set, let $T_Q$ consist of the vertices of $S$ whose
colours are not in $Q$. For fixed $Q$, the set $T_Q$ is an independent $1/4$-sample.
Lemma~\ref{lem:quarter} and a union bound therefore give
\begin{equation}\label{eq:colours}
\Prob\bigl(\text{some }T_Q\text{ is not a connected dominating set of }G\bigr)
\le \binom{2h}{h}\bigl[a(K+1)e^{-(K+1)/8}+2^ae^{-K/49}\bigr].
\end{equation}

Suppose that every $T_Q$ is a connected dominating set of $G$. Each $T_Q$
is nonempty by definition. Hence $\abs S\ge h+1$:
otherwise an $h$-element colour set could contain every colour used on $S$, giving an
empty $T_Q$. Now take $X\subseteq S$ with $\abs X<h$, and choose $Q$ containing all
colours used on $X$. The set $T_Q\subseteq S\setminus X$ is connected and dominates
$G$. Every other vertex of $S\setminus X$ has a neighbour in it, so $G[S\setminus X]$
is connected. This proves $\kappa(G[S])\ge h$, including the order condition. In
particular, every sample $S$ with smaller connectivity forces a failure in
\eqref{eq:colours}, whatever colouring is chosen.

It remains to estimate the right-hand side of~\eqref{eq:colours}. Since $K\ge512a$,
we have $2h\le K/256$ and $a\le K/512$. Also
$$
\log[a(K+1)]\le2\log(K+1)\le K/32\qquad(K\ge512).
$$
Using $\binom{2h}{h}\le2^{2h}$, the first term is at most
$$
\exp\left(\frac{K\log2}{256}+\frac{K}{32}-\frac{K+1}{8}\right)
\le e^{-K/16}.
$$
The second term is at most
$$
\exp\left(\frac{3K\log2}{512}-\frac{K}{49}\right)\le e^{-K/64},
$$
since $1/49-3\log2/512>1/64$. Their sum is at most $2e^{-K/64}$, proving the theorem.
\end{proof}

\begin{proof}[Proof of Theorem~\ref{thm:spanning}]
For $r=1$, we have
$\kappa(H)\ge6000(a+1)\ge a+1\ge\alpha(H)+1$, so the Hamilton-connected
form of the \CE\ theorem gives the required path. For $r\ge2$, we follow the two-part completion
argument of~\cite{AHK}, using Theorem~\ref{thm:sampling} in place of the order-dependent
sampling estimate. Write $b=a+r$.
Delete all $2r$ prescribed vertices from $H$, obtaining $F$, and let $K=5120b$.
Then we have $
\kappa(F)\ge6000b-2r\ge K$ and $\alpha(F)\le a.$
Each prescribed vertex also has at least $K$ neighbours in $F$.

Partition $V(F)$ randomly into two complementary halves $X,Y$, using independent fair
choices for the vertices. Theorem~\ref{thm:sampling}, applied to each half, shows that
$\kappa(F[X]),\ \kappa(F[Y])\ge10b$
except with probability at most $4e^{-K/64}$. No independence between these two
conclusions is needed.

We also require every vertex of $H$ to have many neighbours in both halves. If a vertex
has $d$ neighbours in $F$, then the number $Z$ lying in one specified half is binomial
with parameters $d,1/2$, and
$$
\Prob(Z<d/4)\le2^{d/4}\E(2^{-Z})
=2^{d/4}(3/4)^d\le e^{-d/16}.
$$
For vertices of $F$, the monotonicity of $(d+1)e^{-d/16}$ for $d\ge K\ge512$ and
\eqref{eq:reciprocal} give
$$
\sum_{v\in V(F)}e^{-d_F(v)/16}\le a(K+1)e^{-K/16}.
$$
Treating the $2r$ prescribed vertices separately and considering both halves, the
probability that some vertex receives fewer than one quarter of its $F$-neighbours
in one half is at most
\begin{equation}\label{eq:neighbours}
[2a(K+1)+4r]e^{-K/16}.
\end{equation}
The sum of this quantity and $4e^{-K/64}$ is less than one. Indeed,
$2a(K+1)+4r\le2b(K+3)\le(K+3)^2$, and $2\log(K+3)\le K/32$ for $K\ge512$.
Thus there is a partition such that both induced halves are $10b$-connected and
every vertex of $H$ has at least $K/4\ge10b$ neighbours in each half.

Adding vertices each having at least $s$ neighbours in an $s$-connected graph preserves
$s$-connectivity: after fewer than $s$ deletions the original graph remains connected
and every surviving added vertex still has a neighbour in it. Hence
$H[X\cup\{x_1,y_1,\ldots,x_{r-1},y_{r-1}\}]$ is $10b$-connected. The Thomas--Wollan theorem gives pairwise vertex-disjoint paths
$P_1,\ldots,P_{r-1}$ joining the first $r-1$ prescribed pairs. They avoid all of $Y$
and the last two prescribed vertices.

After deleting the vertices of these paths, the remaining graph still contains the
$10b$-connected graph $F[Y]$. Every other remaining vertex retains at least $10b$
neighbours in $Y$, so this remaining graph is $10b$-connected and has independence
number at most $a$. By the \CE\ theorem it has a Hamilton path from $x_r$ to $y_r$.
Together with $P_1,\ldots,P_{r-1}$, this gives the required spanning family.
\end{proof}

\section{Linking symmetric subdigraphs}

Jackson's linking lemma~\cite[Lemma~5]{Jackson} requires
$k(H_i)>(2m-3)\aTwo(H_i)$ for each noncomplete part. The next lemma replaces
this multiplicative dependence by an additive one. The comparison is in the
parameters, not a uniform improvement of numerical constants. Arcs between
different parts are unrestricted.

\begin{lemma}\label{lem:linking}
Let $D$ be a digraph with a partition $V(D)=V_1\dot\cup\cdots\dot\cup V_m$,
where $m\ge2$ and every $V_i$ is nonempty. For each $i$, let $H_i$ be a
symmetric subdigraph of $D[V_i]$ with $V(H_i)=V_i$. Write
$$
\nu_i=\aTwo(H_i),\qquad \beta=\sum_{i=1}^m\nu_i,\qquad \rho=m-1.
$$
If $k(D)\ge7\beta$ and each $H_i$ is either complete symmetric or satisfies
\begin{equation}\label{eq:local-hyp}
k(H_i)\ge6000(\nu_i+\rho),
\end{equation}
then $D$ has a Hamilton cycle.
\end{lemma}

\begin{proof}
Reindex the symmetric subdigraphs of order at least four as $J_1,\ldots,J_p$,
and the others as $C_1,\ldots,C_s$, so $p+s=m$. Every $C_j$ is complete
symmetric by~\eqref{eq:local-hyp}. Let $S_0=\bigcup_jV(C_j)$,
$q_0=\abs{S_0}\le3s$, and $\mu_i=\aTwo(J_i)$. Since $m\le\beta$ and
$\abs{V(D)}\ge7\beta+1$, we have $p\ge1$.

For a set $U$ of chosen vertices in the large parts, let $Z(U)$ be the union
of the isolated vertices of $J_i-U$, over $1\le i\le p$. Within each part
these vertices form an independent set, so
\begin{equation}\label{eq:isolated}
\abs{Z(U)}\le\sum_{i=1}^p\mu_i=\beta-s.
\end{equation}
Every unchosen large-part vertex outside $Z(U)$ has an unchosen symmetric
neighbour in its own part.

We first fix a small collection of directed paths, complete their arcs to a
cycle factor, and then select a balanced subgraph of its quotient. The selection
will limit the number of local terminal pairs without losing the connections
between the parts. The four claims below carry out these steps and the final
spanning linkage.

\begin{claim}\label{cl:tree}
There are $p-1$ vertex-disjoint directed paths
\begin{equation}\label{eq:four-path}
u\to x\to y\to v,
\end{equation}
where $u,x$ lie in one $J_i$ and $y,v$ lie in another, such that the middle
arcs form an undirected tree on the $p$ large-part indices.
\end{claim}

\begin{proof}[Proof of the claim]
Take a maximal family of such paths whose middle arcs form a forest.
If its size is $j<p-1$, then $j\le p-2$ and its used vertex set $U$ has
size $4j$. By~\eqref{eq:isolated},
$$
\abs{S_0\cup U\cup Z(U)}
\le q_0+4j+\beta-s
\le\beta+4p+2s-8<7\beta.
$$
Every component of the index forest retains a vertex outside this deleted set.
An isolated index corresponds to an unused part. A leaf part of a nontrivial
tree component has exactly two used vertices. If it is complete, its order
is at least four and an edge remains. If it is noncomplete, its connectivity
is at least $10$, so an edge again remains after the two vertices are deleted.
The endpoints of this edge are outside $Z(U)$.

The remaining digraph is strongly connected, and hence has an arc $x\to y$
between different components of the index forest. As $x,y$ avoid $Z(U)$,
they have unchosen symmetric neighbours $u,v$ in their respective parts.
The four vertices are distinct, and this new path extends the family,
a contradiction. When $p=1$, the required family is empty.
\end{proof}

\begin{claim}\label{cl:factor}
The paths of Claim~\ref{cl:tree} extend to a family of $m-1$ vertex-disjoint
paths, covering $S_0$, whose arc set $F_0$ satisfies $\abs{F_0}\le6\beta-3$.
There is a cycle factor $\mathcal F$ of $D$ containing $F_0$.
\end{claim}

\begin{proof}[Proof of the claim]
For each $C_j$, choose a spanning directed path $Q_j$ and extend it to
\begin{equation}\label{eq:small-path}
u_j\to x_j\to Q_j\to y_j\to v_j,
\end{equation}
where $u_j,x_j$ and $y_j,v_j$ are fresh symmetric neighbouring pairs in
large parts. The endpoint parts may coincide, but the four vertices are
distinct. After $j<s$ small parts have been treated, $4(p-1+j)$ large-part
vertices have been used. To choose an in-neighbour of the first vertex of
the next $Q$, exclude $S_0$, the used set $U$, and $Z(U)$. Once the first
pair is chosen, the analogous forbidden set for an out-neighbour of the
last vertex has size at most
$$
q_0+4(p-1+j)+2+\beta-s\le\beta+4m+2s-6<7\beta.
$$
Both choices follow from the minimum in- and outdegree of $D$, and
\eqref{eq:isolated} supplies their symmetric neighbours. Thus all extensions
can be made disjointly.

Let $F_0$ be the union of the arcs in~\eqref{eq:four-path}
and~\eqref{eq:small-path}. There are $(p-1)+s=m-1$ paths, and
\begin{equation}\label{eq:protected-size}
\abs{F_0}=3(p-1)+\sum_{j=1}^s(\abs{C_j}+3)
=3m+q_0-3\le6\beta-3.
\end{equation}
An arcless set in $D$ meets $V(H_i)$ in at most $\nu_i$ vertices, so
$\aZero(D)\le\beta$. Since $F_0$ has at most one entering and one leaving
arc at each vertex, Lemma~\ref{lem:factor} gives the claimed cycle factor.
\end{proof}

\begin{claim}\label{cl:quotient}
There are vertex-disjoint actual directed paths covering $S_0$, with mutually
distinct endpoints in the large parts, whose path quotient is connected and
balanced. If $w_i$ counts their internal vertices in $J_i$ and $\ell_i$ counts
their initial vertices there, then there are also $\ell_i$ terminal vertices
in $J_i$, and
\begin{equation}\label{eq:transit}
\ell_i\ge1,\qquad w_i+\ell_i\le\rho.
\end{equation}
\end{claim}

\begin{proof}[Proof of the claim]
In $\mathcal F$, replace each subpath $x_jQ_jy_j$ by a labelled arc
$x_j\to y_j$. All remaining vertices lie in large parts and still have
indegree and outdegree one. Form a directed multigraph $M$ on their $p$
part indices: retain all ordinary factor arcs between distinct large parts
and all labelled arcs, and omit ordinary internal arcs. Labelled arcs may
be loops of $M$, and parallel arcs retain their distinct identities.
The multigraph $M$ is balanced, since summing the factor indegrees and
outdegrees over any part gives the same total.

Mark the $p-1$ middle arcs of Claim~\ref{cl:tree} and the $s$ labelled arcs,
for a total of $\rho=m-1$ marks. Every marked tail has its factor predecessor
in its own part, and every marked head has its factor successor in its own
part. These internal arcs are omitted from $M$; see Figure~\ref{fig:guards}.

\begin{figure}[htbp]
\centering
$$
\underbrace{u\longrightarrow x}_{\text{inside }J_i}
\ \xrightarrow{\ \text{marked arc}\ }\ 
\underbrace{y\longrightarrow v}_{\text{inside }J_j}
$$
\caption{The internal arcs adjacent to a marked arc are guards. They are not
selected in the quotient. Thus $x$ has no selected incoming arc and $y$ has
no selected outgoing arc. For a labelled arc, its entire small-part path
occupies the middle position.}
\label{fig:guards}
\end{figure}

Our aim is to keep balance and all marked connections while selecting at
most $\rho$ entering and leaving incidences at each large part. Decompose
the arcs of $M$ into directed cycles, each with no repeated vertex, allowing
a loop as a one-arc cycle. Retain only cycles containing a marked arc.
At most $\rho$ cycles remain, so every part has at most $\rho$ selected
incoming and at most $\rho$ selected outgoing incidences.

Now return to the actual large-part vertices, temporarily treating labelled
paths as single arcs. Each vertex has selected indegree and outdegree at
most one, so every nonempty selected component is a path or a cycle.
A selected cycle contains no marked arc: the tail of a marked arc has no
selected incoming arc. Withdraw the arcs of all these selected cycles.
This preserves balance on the part indices and all marks. No labelled arc
is withdrawn. Their vertices are left for the later local path construction,
not deleted from $D$.

Expand the labelled arcs. The resulting actual paths cover $S_0$ and have
mutually distinct endpoints in the large parts. At $J_i$, an internal vertex
contributes one selected incoming and one selected outgoing incidence; an
initial vertex contributes only the latter, and a terminal vertex only the
former. Balance gives the same number $\ell_i$ of initial and terminal
vertices, and the incidence bound gives $w_i+\ell_i\le\rho$.
In particular, a selected head which is also a selected tail is one internal
vertex, not two terminals.

Define the path quotient by replacing each of these paths by an arc from
its initial part to its terminal part; loops and parallel arcs are allowed.
Balance follows from the equal numbers of initial and terminal vertices in
each part. Every marked tree arc remains a whole path at the labelled level,
because its tail has no selected incoming arc and its head has no selected
outgoing arc. Hence the path quotient contains the marked tree and is
connected. If $p=1$, then $s\ge1$ and at least one marked labelled loop
remains. Every part therefore has positive degree, giving $\ell_i\ge1$.
\end{proof}

\begin{claim}\label{cl:spanning}
The paths of Claim~\ref{cl:quotient} extend to a Hamilton cycle of $D$.
\end{claim}

\begin{proof}[Proof of the claim]
Take an Euler tour of the connected balanced path quotient. At each $J_i$,
it prescribes $\ell_i$ pairs, joining the terminal vertex of an entering
path to the initial vertex of the next leaving path. These are $2\ell_i$
distinct vertices. Let $W_i$ be the set of the $w_i$ internal vertices in
$J_i$ already covered by the external paths.

If $J_i$ is complete symmetric, the prescribed pairs have a spanning
disjoint path system in $J_i-W_i$: assign all nonterminal vertices to one
path and use direct arcs for the other pairs. Otherwise, by
\eqref{eq:local-hyp} and~\eqref{eq:transit},
$$
k(J_i-W_i)\ge6000(\mu_i+\rho)-w_i
\ge6000(\mu_i+\ell_i).
$$
Also $\aTwo(J_i-W_i)\le\mu_i$. Apply Theorem~\ref{thm:spanning} to the
undirected graph of this symmetric subdigraph. The resulting paths are
vertex-disjoint, meet the prescribed pairs, and cover $J_i-W_i$. Orient
them from each entering endpoint to its paired leaving endpoint.

Concatenate these local paths and the external paths in Euler-tour order.
Every small-part vertex and every vertex of a $W_i$ lies on one external
path. Every other vertex lies on exactly one local path. The concatenation
is therefore one directed cycle covering all vertices.
\end{proof}

Claim~\ref{cl:spanning} proves the lemma.
\end{proof}

Since every $\nu_i\ge1$, we have $\nu_i+m-1\le\beta$. Thus it is enough in
Lemma~\ref{lem:linking} to require $k(H_i)\ge6000\beta$ for every noncomplete
part. The complete-part case gives the following separate consequence.

\begin{corollary}\label{cor:cliques}
Let $D$ be a digraph whose vertex set is partitioned into $m\ge1$ nonempty
complete symmetric subdigraphs. If $k(D)\ge7m$, then $D$ has a Hamilton cycle.
\end{corollary}

\begin{proof}
For $m\ge2$, apply Lemma~\ref{lem:linking} with $\nu_i=1$ for every $i$, so
$\beta=m$. For $m=1$, the graph is complete symmetric and has at least eight
vertices, hence has a Hamilton cycle.
\end{proof}

Jackson~\cite[Corollary~5]{Jackson} proved this conclusion under
$k(D)>m(m-1)$. Corollary~\ref{cor:cliques} replaces the quadratic dependence
on $m$ by a linear one. With the stated constants, its threshold is strictly
smaller for every integer $m\ge8$; no numerical improvement for all small
$m$ is asserted.

\section{Proof of the quadratic bound}

We include the fixed-threshold decomposition from~\cite{AS}, which is a variant of
Jackson's decomposition lemma~\cite{Jackson}.

\begin{lemma}\label{lem:decomposition}
Let $G$ be a nonempty undirected graph with $\alpha(G)\le a$, and let $B\ge0$ be an
integer. There are a set $S\subseteq V(G)$ and nonempty vertex-disjoint induced
subgraphs $F_1,\ldots,F_t$ covering $V(G)\setminus S$ such that
\begin{enumerate}
\item $t\le a$;
\item $\sum_{i=1}^t\alpha(F_i)\le a$;
\item $\abs S\le(t-1)B$;
\item every $F_i$ is complete or has $\kappa(F_i)>B$.
\end{enumerate}
\end{lemma}

\begin{proof}
Start with $G$ and $S=\varnothing$. Whenever a current part $F$ is noncomplete and
satisfies $\kappa(F)\le B$, choose a vertex cut $R$ of $F$ with $\abs R\le B$;
if $F$ is disconnected, take $R=\varnothing$. Move $R$ to $S$ and replace $F$ by
two nonempty induced subgraphs whose vertex sets are unions of components of $F-R$.
There are no edges between the two new parts, so their independence numbers have
sum at most $\alpha(F)$. Thus the sum of the independence numbers of all current
parts never increases. Each nonempty part has independence number at least one,
so at most $a$ parts occur. The procedure terminates after at most $a-1$ splits.
If the final number of parts is $t$, exactly $t-1$ splits occurred, and
$\abs S\le(t-1)B$. At termination every part is complete or has connectivity
greater than $B$.
\end{proof}

We also need the path contraction estimate from~\cite{AS}. A complete proof is included
because the entering and leaving endpoints of a contracted path are different.

\begin{lemma}\label{lem:contraction}
Let $P_1,\ldots,P_q$ be vertex-disjoint directed paths in $D$, with union $S$, and
write $P_j=s_j\cdots t_j$. Replace each $P_j$ by one new vertex $p_j$. On
$V(D)\setminus S$ keep the old arcs, and define
\begin{align*}
u\to p_j\text{ in }D^*&\quad\Longleftrightarrow\quad u\to s_j\text{ in }D,\\
p_j\to u\text{ in }D^*&\quad\Longleftrightarrow\quad t_j\to u\text{ in }D,\\
p_i\to p_j\text{ in }D^*&\quad\Longleftrightarrow\quad t_i\to s_j\text{ in }D
\quad(i\ne j).
\end{align*}
Loops are ignored. If $k(D)>\abs S$, then $k(D^*)\ge k(D)-\abs S$. Every Hamilton
cycle of $D^*$ expands to a Hamilton cycle of $D$.
\end{lemma}

\begin{proof}
Let $\lambda=k(D)$ and let $X\subseteq V(D^*)$ have $\abs X<\lambda-\abs S$.
Take distinct vertices $x,y\in V(D^*)\setminus X$. Represent $x$ by $t_i$ if
$x=p_i$, and by $x$ otherwise; call this representative $\widehat x$. Represent
$y$ by $s_j$ if $y=p_j$, and by $y$ otherwise; call it $\widehat y$. The representatives
are distinct because the original paths are disjoint and $x\ne y$.

Delete from $D$ the set
$$
Y=(S\setminus\{\widehat x,\widehat y\})
\cup\bigl(X\cap(V(D)\setminus S)\bigr).
$$
Its size is at most $\abs S+\abs X<\lambda$. Thus $D-Y$ has a directed
$\widehat x\widehat y$-path. Choose it simple. It meets $S$ only, possibly, at its
two endpoints, so the definition of $D^*$ turns it into a directed $xy$-path
avoiding $X$. Hence $D^*-X$ is strongly connected. Moreover
$$
\abs{V(D^*)}=\abs{V(D)}-\abs S+q\ge \lambda-\abs S+1,
$$
which verifies the order condition. Finally, replace each occurrence of $p_j$ on
a Hamilton cycle of $D^*$ by the entire path $s_j\cdots t_j$. The definitions of
all entering and leaving arcs guarantee a Hamilton cycle of $D$.
\end{proof}

\begin{proof}[Proof of Theorem~\ref{thm:main}]
If $a=1$, the hypothesis gives $k(D)\ge12000$ and hence
$\abs{V(D)}\ge12001$. Every pair of distinct vertices forms a $2$-cycle,
so $D$ is a complete symmetric digraph and has a Hamilton cycle. Suppose $a\ge2$. We prove the stronger bound~\eqref{eq:precise-main}.
Let
$$
B=12000a-1,\qquad K_0=14a+(a-1)B,
$$
and let $D$ be $K_0$-connected with $\aTwo(D)\le a$.

Apply Lemma~\ref{lem:decomposition} to $G(D)$. We obtain $S,F_1,\ldots,F_t$ such that
\begin{equation}\label{eq:final-parts}
t\le a,\qquad \sum_{i=1}^t\alpha(F_i)\le a,
\qquad \abs S\le(t-1)B,
\end{equation}
and each $F_i$ is complete or satisfies $\kappa(F_i)>B$. For each $i$, let $A_i$
be the symmetric subdigraph of $D$ whose vertex set is $V(F_i)$ and whose arcs are
exactly the two opposite arcs of the edges of $F_i$. Thus
$$
G(A_i)=F_i,\qquad \aTwo(A_i)=\alpha(F_i),\qquad k(A_i)=\kappa(F_i).
$$
We take only these $2$-cycle arcs; the induced digraph $D[V(F_i)]$ may contain
additional one-way arcs and need not be symmetric.

If $t=1$, then no split occurred and $S=\varnothing$. The symmetric subdigraph
$A_1$ spans $D$. It is either complete, or has
$k(A_1)\ge B+1=12000a\ge\aTwo(A_1)+1$; in either case it has a Hamilton cycle.
We may therefore assume $t\ge2$.

By Gallai--Milgram, $D[S]$ has a path cover of size $q\le a$, since
$\aZero(D[S])\le\aTwo(D[S])\le a$. If $S=\varnothing$, take $q=0$.
Contract these paths as in Lemma~\ref{lem:contraction}, obtaining $D^*$. From
\eqref{eq:final-parts},
\begin{equation}\label{eq:final-connectivity}
k(D^*)\ge k(D)-\abs S
\ge K_0-(a-1)B=14a.
\end{equation}

The symmetric subdigraphs $A_1,\ldots,A_t$ and the $q$ new singleton vertices
partition $V(D^*)$. Write
$$
\nu_i=\aTwo(A_i),\qquad m=t+q,\qquad
\beta^*=\sum_{i=1}^t\nu_i+q\le2a.
$$
Then $m\ge2$, and~\eqref{eq:final-connectivity} gives $k(D^*)\ge7\beta^*$.
For every noncomplete $A_i$, its internal arcs are unchanged by the contraction,
and
$$
k(A_i)\ge B+1=12000a\ge6000\beta^*
\ge6000(\nu_i+m-1).
$$
The last inequality holds because every other nonempty part contributes at least
one to $\beta^*$. All new singleton parts are complete symmetric. Thus
Lemma~\ref{lem:linking} applies and $D^*$ has a Hamilton cycle. Expanding the
contracted paths gives a Hamilton cycle of $D$. Finally,
$$
K_0=12000a^2-11987a+1\le12000a^2,
$$
which proves the theorem.
\end{proof}

The local estimate in Theorem~\ref{thm:spanning} and the global requirement in
Lemma~\ref{lem:linking} are both linear in their respective parameters. The
quadratic term in Theorem~\ref{thm:main} comes from the at most $a-1$ vertex cuts
of order at most $12000a-1$ in Lemma~\ref{lem:decomposition}. Thus this proof does
not give a linear bound for $f_2$.

\section*{AI-assisted tools statement}
AI-assisted tools (ChatGPT 6 pro) were used during the preparation of this manuscript
for mathematical brainstorming, checking intermediate arguments, literature review assistance, and language refinement. All final mathematical ideas, proofs, and claims were independently checked and
approved by the authors, who take full responsibility for the content
of the manuscript.

\end{document}